\documentclass[sn-mathphys]{sn-jnl}

\jyear{2021}%

\usepackage{amsmath,amssymb,amsthm}

\newcommand{\supp}{\mathrm{supp\,}} 
\newcommand{\sinc}{\mathrm{sinc}} 

\renewcommand{\R}{\mathbb R}

\newcommand{\N}{\mathbb{N}}

\theoremstyle{thmstyleone}%
\newtheorem{theorem}{Theorem}
\newtheorem*{nonumber-theorem}{Theorem}

\newtheorem{lemma}[theorem]{Lemma}
\newtheorem{proposition}[theorem]{Proposition}%

\theoremstyle{thmstyletwo}%
\newtheorem{remark}{Remark}%

\theoremstyle{thmstylethree}%
\newtheorem{definition}{Definition}%

\begin{document}

\title[Unexpected behaviors of operators with small time-frequency dispersion]{Some unexpected behaviors of operators with small time-frequency dispersion}


\author*{\fnm{Dae Gwan} \sur{Lee}}\email{daegwans@gmail.com, dae.lee@ku.de} 

\affil{\orgdiv{Mathematisch-Geographische Fakult\"at}, \orgname{Katholische Universit\"at Eichst\"att-Ingolstadt}, \orgaddress{\city{Eichst{\"a}tt}, \postcode{85071}, \country{Germany}}}


\abstract{
We study the approximation properties of pseudo-differential operators with small time-frequency dispersion, meaning that their spreading functions are supported in a small neighborhood of the origin.
It is commonly assumed that for such operators $H$, the output $Hf$ can differ only a little from a scalar multiple of the input $f$. However, we disprove this heuristic statement, hence revealing some unexpected behaviors of such operators. 
}

\keywords{Pseudo-differential operators, Spreading function, Kohn-Nirenberg symbol, Time-frequency analysis, Superoscillations}
 
\pacs[MSC Classification]{47G30\!\!, 47B38\!\!, 94A20\!\!, 42B35}

\maketitle

\section{Introduction and Main Results}
\label{sec:intro}

A time-varying communication channel is often modeled as a pseudo-differential operator $H$ of the form (see e.g., \cite[Section 2]{KP14}, \cite{PW06,PW16,St06})
\begin{equation}\label{equ:IntroSystem}
	(H f)(x)
= \iint_{\R^2} \eta_H (t,\nu) \, M_{\nu} T_t f (x) \, dt \, d\nu ,
\end{equation}
where $T_{t}$ is \emph{translation (time shift)} by $t \in \R$, i.e., $T_{t} f (x) = f(x-t)$, and $M_{\nu}$ is \emph{modulation (frequency shift)} by $\nu \in \R$, i.e., $M_{\nu} f (x) = e^{2 \pi i \nu x} \, f(x)$; time shifts/delays are caused by distance of transmission paths while frequency shifts are caused by movements of the transmitter and receiver, namely by the Doppler effect.
The spreading function $\eta_H (t,\nu)$ serves as a weight function for time-frequency shifts $M_{\nu} T_t$, hence, the operator $H$ can be viewed as a superposition of time-frequency shifts. For instance, if $\eta_H$ is the Dirac delta supported at $(t_0,\nu_0)$, then $H$ is simply the time-frequency shift operator $M_{\nu_0} T_{t_0}$; in particular, $H$ is the identity operator on $L^2(\R)$ if $\eta_{H} (t,\nu) = \delta(t) \, \delta(\nu)$.

Since implementing a general operator of the form \eqref{equ:IntroSystem} would require infinite resource,
it is natural to consider operators with limited time-frequency dispersion.
The space of bounded linear operators on $L^2(\R)$ whose spreading function vanishes outside a set $S \subset \R^2$, called the \emph{operator Paley-Wiener space} for $S$,
is given by
\[
OPW (S) = \{ H \in \mathcal{L} (L^2(\R)) : \supp \eta_H \subset S \} .
\]

In this note, we focus on the case where $S$ is a set of small measure.
As a motivating example, let $S = [-\alpha,\alpha]^2$ with $0 < \alpha < 1$ and consider
the functions $f = \chi_{[-\frac{1}{2},\frac{1}{2}]}$ and $y = T_{N} f$ with a large $N \geq 1 {+} \alpha$.
Then for any $H \in OPW(S)$, we have $\supp Hf \subset [-\frac{1}{2}{-}\alpha,\frac{1}{2}{+}\alpha]$ and thus
$\| Hf - y \|_{L^2}^2 \geq \int_{N-1/2}^{N+1/2} \vert 0 - 1 \vert^2 \, dx = 1$;
in particular, there is no $H \in OPW(S)$ satisfying $Hf = y$.
This shows that if the input $f$ and output $y$ differ by a long time delay (more generally, by a time-frequency shift of a large factor), then
there is no operator $H \in OPW([-\alpha,\alpha]^2)$ with $\alpha > 0$ small, such that $Hf$ is close to $y$.
Note that as $\alpha \rightarrow 0$, the space $OPW([-\alpha,\alpha]^2)$ reduces in a rough sense to $OPW(\{ (0,0) \})$ which consists of scalar multiples of the identity operator on $L^2(\R)$.
It is therefore tempting to make the following heuristic statement.
\begin{quote}
\it
If $\alpha > 0$ is small, the output $Hf$ of $H \in OPW([-\alpha,\alpha]^2)$ can differ only a little from a scalar multiple of the input $f \in L^2(\R)$.
\end{quote}
We will show that this statement is actually far from true.
Specifically, we prove that even if $\alpha > 0$ is very small,
it is possible to construct an operator $H$ in $OPW([-\alpha,\alpha]^2)$ and an input signal $\chi_{[-B,B]}$ (the characteristic function of $[-B,B]$), such that the output $H \chi_{[-B,B]}$ approximates any given function $y \in L^2(\R)$ in the $L^2$-norm, that is, $\| H \chi_{[-B,B]} - y \|_{L^2}$ is arbitrarily small.
In fact, we construct a \emph{Hilbert-Schmidt} operator in $OPW([-\alpha,\alpha]^2)$ satisfying such a property.

\begin{theorem}\label{thm:existsHSOperator-input-chi}
Let $\alpha > 0$ and $\beta > 0$. For any $\epsilon > 0$ and $y \in L^2(\R) \backslash \{ 0 \}$,
there exist a constant $B > 0$ and a Hilbert-Schmidt operator $H$ in $OPW ([-\alpha,\alpha] {\times} [-\beta,\beta])$ such that $\| H \chi_{[-B,B]} - y \|_{L^2} < \epsilon$.
Moreover, one can choose any $B > 0$ satisfying $\int_{[-B,B]^C} \vert y(x) \vert^2 \, dx < c \, \epsilon$ for some $0 < c < 1$; in particular, if $\supp y \subset [-M,M]$, then $B$ can be chosen to be $M$.
\end{theorem}

It is easily seen that Theorem \ref{thm:existsHSOperator-input-chi} is equivalent to: For each $\alpha > 0$, the set $\{ H \chi_{[-N,N]} : H \in OPW ([-\alpha,\alpha]^2) \cap HS(L^2(\R)) , \; N \in \N \}$ is dense in $L^2(\R)$.

By interchanging the roles of time and frequency domains, we obtain the following analogue of Theorem \ref{thm:existsHSOperator-input-chi}. 
For $B > 0$, let $\phi_B (x) := \sin(2 \pi B x) / (\pi x)= 2B \, \sinc (2B x)$; it is easy to check that $\widehat{\phi}_B = \chi_{[-B,B]}$, the characteristic function of $[-B,B]$.

\begin{theorem}\label{thm:existsHSOperator-input-Sinc}
Let $\alpha > 0$ and $\beta > 0$. For any $\epsilon > 0$ and $y \in L^2(\R) \backslash \{ 0 \}$,
there exist a constant $B > 0$ and a  Hilbert-Schmidt operator $H$ in $OPW ([-\alpha,\alpha] {\times} [-\beta,\beta])$ such that $\| H \phi_B - y \|_{L^2} < \epsilon$.
Moreover, one can choose any $B > 0$ satisfying $\int_{[-B,B]^C} \vert\widehat{y}(\xi)\vert^2 \, d\xi < c \, \epsilon$ for some $0 < c < 1$; in particular, if $y$ is bandlimited to $[-M,M]$ (i.e., $\supp \widehat{y} \subset [-M,M]$), then $B$ can be chosen to be $M$.
\end{theorem}

Theorems \ref{thm:existsHSOperator-input-chi} and \ref{thm:existsHSOperator-input-Sinc} rely on the theory of superoscillations, which is rooted on the fact that bandlimited functions can oscillate substantially faster than its maximum frequency, on arbitrarily long but finite intervals\footnote{This fact is also manifested by the theory of prolate spheroidal wave functions; see e.g., \cite[Section 2.3]{Da92} and \cite[Section I.A]{FK06}.}. 
The phenomenon of superoscillations has recently drawn great interest in the field of quantum mechanics and in applications such as superresolution imaging.
For more details, see \cite{ACSST17,FK06,FK02,LF14-deriv,LF14-energy} and references therein.

Recently, Alpay et al.~\cite{ACS21} discussed
some applications of superoscillations in harmonic analysis, such as extension of positive definite functions and interpolation of polynomials.
As an another application of superoscillations, our main result reveals some unexpected behaviors of pseudo-differential operators with small time-frequency dispersion.

\section{Preliminaries}
\label{sec:prelim}

\begin{definition}\label{def:paleywienerspaces}
For a measurable set $S\subset \R$, we define the \emph{Paley-Wiener space}
\[
PW(S) = \{ f \in L^2(\R) : \supp \widehat{f} \subset S\} ,
\]
where $\widehat f(\xi) = \mathcal{F}f (\xi) := \int f(x)\, e^{-2\pi i x\xi}\,dx$ denotes the Fourier transform of $f$.
In other words, the {\em Paley-Wiener space} for $S$
consists of all square-integrable functions on $\R$ which are bandlimited to $S$.
\end{definition}

Using the Schwartz kernel theorem, one can show that every bounded linear operator $H: L^2(\R) \rightarrow L^2(\R)$ can be represented in the form (see e.g., \cite{Gro01}, \cite[Section 2]{KP14})
\begin{equation}\label{eqn:operator-H-eta-sigma}
H f (x)
= \iint \eta_H (t,\nu) \, e^{2 \pi i x \nu} \, f(x-t) \, dt \, d\nu
= \int \sigma_H (x,\xi) \, \widehat{f} (\xi) \, e^{2 \pi i x \xi} \, d\xi ,
\end{equation}
where the operator $H$ is in 1-1 correspondence with its spreading function $\eta_H$, and equivalently with its Kohn-Nirenberg symbol $\sigma_H$. Formally, it holds
\begin{equation}\label{eqn:eta-sigma-relation}
\begin{split}
&\eta_H (t,\nu) = \iint \sigma_H (x,\xi) \, e^{- 2 \pi i (x \nu - t \xi )} \, dx \, d\xi , \\
&\sigma_H (x,\xi) = \iint \eta_H (t,\nu) \, e^{- 2 \pi i (t \xi - x \nu)} \, dt \, d\nu ,
\end{split}
\end{equation}
that is, $\eta_H = \mathcal F_s \sigma_H$ and $\sigma_H = \mathcal F_s \eta_H$,
where $\mathcal F_s g (t,\nu) := \iint g (x,\xi) \, e^{- 2 \pi i (x \nu - t \xi )} \, dx \, d\xi$ is the \emph{symplectic Fourier transform}\footnote{It is easily seen that $\mathcal F_s^{-1} = \mathcal F_s$, whereas $\mathcal{F}^{-1} = \mathcal{F}^3$.} of $g$.
If $H \in \mathcal{L} (L^2(\R))$ is a Hilbert-Schmidt operator, then
$\| H \|_{HS}
	= \| \eta_H \|_{L^2(\R^2)}
	= \| \sigma_H \|_{L^2(\R^2)}$.

Analogous to Paley-Wiener spaces consisting bandlimited functions, we define classes of operators with bandlimited Kohn-Nirenberg symbols. 
Note that $\supp \eta_H = \supp \mathcal F_s \sigma_H$.

\begin{definition}[Definitions 5 and 29 in \cite{PW16}; Definition 4.1 in \cite{Pf13}]
\label{def:operatorpaleywienerspaces}
For a measurable set $S\subset \R^2$, we define the {\em operator Paley-Wiener spaces}
\[
\begin{split}
	OPW(S) &= \{H\in \mathcal{L} (L^2(\R)) : \supp\eta_H \subset S\} , \\
	OPW^{2,\infty}(S) &= \{H\in \mathcal{L} (L^2(\R)) : \supp\eta_H \subset S, \; \| \sigma_H \|_{L^{2,\infty}} < \infty \} , \\		
	OPW^{\infty,2}(S) &= \{H\in \mathcal{L} (L^2(\R)) : \supp\eta_H \subset S, \; \| \sigma_H \|_{L^{\infty,2}} < \infty \} , \\	
	OPW^2 (S) &=  \{H\in \mathcal{L} (L^2(\R)) : \supp\eta_H \subset S, \; \sigma_H \in L^2(\R^2) \\
	& \;\;\,\qquad\qquad\qquad\qquad\qquad\qquad\qquad
	\textnormal{or equivalently,} \; \eta_H \in L^2(\R^2) \} ,
\end{split}
\]
where $\mathcal{L} (L^2(\R))$ denotes the class of bounded linear operators on $L^2(\R)$, and
\[
\begin{split}
	\| \sigma_H \|_{L^{2,\infty}} &= \left\| \left( \int \vert \sigma_H ( x, \cdot )\vert^2 \, dx \right)^{1/2} \right\|_{\infty} , \\
	\| \sigma_H \|_{L^{\infty,2}} &= \left( \int \| \sigma_H ( \cdot , \xi ) \|_{\infty}^2 \, d\xi \right)^{1/2}
\end{split}
\]
are the standard mixed $L^{p,q}$-norms, see e.g.,  \cite[Definition 11.1.2]{Gro01}.
\end{definition}

Clearly, the spaces $OPW^{2,\infty}(S)$, $OPW^{\infty,2}(S)$, $OPW^2 (S)$ are subspaces of $OPW (S)$, and in particular, $OPW^2(S)$ is the space of all Hilbert-Schmidt operators in $OPW(S)$, i.e., $OPW^2 (S) = OPW(S) \cap HS(L^2(\R))$.
Moreover, $OPW^{2,\infty}(S)$, $OPW^{\infty,2}(S)$, $OPW^2 (S)$ are Banach spaces with norm $\| H \|_{OPW^{\infty,2}} = \| \sigma_H \|_{L^{2,\infty}}$, $\| H \|_{OPW^2} = \| \sigma_H \|_{L^2}$, and $\| H \|_{OPW^{\infty,2}} = \| \sigma_H \|_{L^{\infty,2}}$, respectively.

As we shall see below, the space $OPW^{2,\infty}(S)$ contains multiplications by bandlimited functions while $OPW^{\infty,2}(S)$ contains convolutions with compactly supported kernels.

\medskip

\noindent
\textbf{Multiplication operators}. \
If $H \in \mathcal{L} (L^2(\R))$ is a multiplication operator with multiplier $m \in L^2(\R)$, then
\[
H f (x) = m (x) \, f(x) = \int m (x) \, \widehat{f} (\xi) \, e^{2 \pi i x \xi} \, d\xi
\quad \text{for all} \;\; f \in L^2(\R) ,
\]
and thus $\sigma_H (x,\xi) = m (x)$ by \eqref{eqn:operator-H-eta-sigma}.
Conversely, if $H \in \mathcal{L} (L^2(\R))$ with $\sigma_H (x,\xi) = m (x)$ for some $m \in L^2(\R)$, then $H$ is the multiplication operator with multiplier $m$.
In this case, we have $\| \sigma_H \|_{L^{2,\infty}} = \| m \|_{L^2} < \infty$ and
the relation \eqref{eqn:eta-sigma-relation} yields $\eta_H (t,\nu) = \delta(t) \, \widehat{m} (\nu)$, so we have $H \in OPW^{2,\infty}(S)$ if and only if $\supp \widehat{m} \subset S_\nu$ if and only if
$m \in PW(S_\nu)$, where $S_\nu \subset \R$ is the segment of the $\nu$-axis (vertical axis) intersecting $S$.

\medskip

\noindent
\textbf{Convolution operators}. \
If $H \in \mathcal{L} (L^2(\R))$ is a convolution operator with kernel $h \in L^2(\R)$, then for any $f \in L^2(\R)$ we have $\widehat{Hf} (\xi) = \widehat{h} (\xi) \, \widehat{f} (\xi)$ and thus, $\sigma_H (x,\xi) = \widehat{h} (\xi)$ by \eqref{eqn:operator-H-eta-sigma}.
Conversely, if $H \in \mathcal{L} (L^2(\R))$ with $\sigma_H (x,\xi) = \widehat{h} (\xi)$ for some $\widehat{h} \in L^2(\R)$, then $H$ is the convolution operator with kernel $h$.
In this case, we have $\| \sigma_H \|_{L^{\infty,2}} = \| \widehat{h} \|_{L^2}  < \infty$ and $\eta_H (t,\nu) = h (t) \, \delta(\nu)$ by \eqref{eqn:eta-sigma-relation}, hence, we have $H \in OPW^{\infty,2}(S)$ if and only if
$\supp h \subset S_t$, where $S_t \subset \R$ is the segment of the $t$-axis (horizontal axis) intersecting $S$.

\begin{remark}\label{rmk:prop:existsMultiplicationOperator-input-chi}
\rm
When $S \subset \R^2$ is compact, the domain of $H \in OPW (S)$ can be extended to the space of tempered distributions.  
This allows one to develop a sampling theory for operators in $OPW (S)$ in close analogy with the classical sampling theory for bandlimited functions.
See \cite{DG11,HB13,LPP19,MBH13,Pf13,PW06,PW16,WPK15} and references therein.
\end{remark}

\begin{proposition}[Theorem 4.1 in \cite{KP14}]
\label{prop:Thm4-1-in-KP14}
If $S \subset \R^2$ is a compact set, then there exist constants $0 < A \leq B < \infty$ such that
$A \, \| \sigma_H \|_{L^\infty(\R^2)}
\leq \| H \|_{L^2(\R) \rightarrow L^2(\R)}
\leq B \, \| \sigma_H \|_{L^\infty(\R^2)}$
for all $H \in OPW (S)$.
\end{proposition}

\section{Proof of Theorem \ref{thm:existsHSOperator-input-chi}}
\label{sec:thm:existsHSOperator-input-chi}

\begin{proposition}\label{prop:existsMultiplicationOperator-input-chi}
Let $\alpha , \epsilon > 0$, $0 < c < 1$, and $y \in L^2(\R) \backslash \{ 0 \}$.
There exist a constant $B > 0$ and a multiplication operator $H \in \mathcal{L} (L^2(\R))$ with multiplier $m \in PW[-\alpha,\alpha]$, i.e., $Hf(x) = m (x) \, f(x)$ for all $f \in L^2(\R)$,
such that $\| H \chi_{[-B,B]} - y \|_{L^2} < \frac{1}{2} (1+c) \epsilon$.
Moreover, one can choose any $B > 0$ satisfying $\int_{[-B,B]^C} \vert y(x) \vert^2 \, dx < c \, \epsilon$; in particular, if $\supp y \subset [-M,M]$, then $B$ can be chosen to be $M$.
\end{proposition}

\begin{proof}
Since $y \in L^2(\R)$, one can choose a sufficiently large $B > 0$ such that
$\int_{[-B,B]^C} \vert y(x) \vert^2 \, dx < c \, \epsilon$.
If $H \in \mathcal{L} (L^2(\R))$ is a multiplication operator with multiplier $m \in L^2(\R)$, then
\[
\begin{split}
\| H \chi_{[-B,B]} - y \|_{L^2}^2
&= \int \vert m (x) \, \chi_{[-B,B]} (x) - y(x) \vert^2 \, dx \\
&= \int_{-B}^B \vert m (x) - y(x) \vert^2 \, dx + \int_{[-B,B]^C} \vert y(x) \vert^2 \, dx \\
&= \int_{-B}^B \vert m (x) - y(x) \vert^2 \, dx + c \, \epsilon . 
\end{split}
\]
Therefore, it is enough to find a function $m \in PW[-\alpha,\alpha]$ satisfying
\begin{equation}\label{eqn:m-approximates-y}
\int_{-B}^B \vert m (x) - y(x) \vert^2 \, dx <  \tfrac{1}{2} (1-c) \epsilon .
\end{equation}
Note that $H$ is then a bounded operator from $L^2(\R)$ to $L^2(\R)$ since every functions in $PW[-\alpha,\alpha]$ are bounded\footnote{Indeed, it holds for all $g \in PW[-\alpha,\alpha]$ and $x \in \R$ that
$\vert g(x) \vert^2 \leq \vert \int_{-\alpha}^{\alpha} \widehat{g} (\xi) \, e^{-2\pi i x\xi} \, dx \vert
\leq 2 \alpha \int_{-\alpha}^{\alpha} \vert\widehat{g} (\xi)\vert^2 \, dx
= 2 \alpha \, \| \widehat{g} \|_{L^2}^2
= 2 \alpha \, \| g  \|_{L^2}^2$.}.
It is known that on a bounded interval $I$, one can approximate any continuous function, in particular, a rapidly oscillating function, by a bandlimited function with a given bandwidth. This can be done, for example, by prescribing the function values and derivative values on a fine grid inside the interval $I$ \cite{FK06,LF14-deriv}.
In this way, one can approximate any continuous function on $I$ in the $L^\infty$-norm and thus in the $L^2$-norm.
Further, since the space of all test functions $C_c^{\infty} (I)$ is dense in $L^2(I)$, it follows that every function in $L^2 (I)$ can be approximated by a bandlimited function with a given bandwidth.
Hence, there exists a function $m \in PW [-\alpha,\alpha]$ satisfying $\int_{-B}^B \vert m (x) - y(x) \vert^2 \, dx <  \frac{1}{2} (1-c) \epsilon$. This completes the proof.
 \end{proof}

\begin{remark}\label{rmk:prop:existsMultiplicationOperator-input-chi}
\rm
(a) The input function $f = \chi_{[-B,B]}$ can be replaced by any function $g \in L^2(\R)$ with $\supp g = [-B,B]$ and $\vert g(x) \vert \geq C > 0$ for a.e.~$x \in [-B,B]$. In that case, one simply needs to replace $y(x)$ with $\frac{y(x)}{g(x)}$ in the proof above. \\
(b) If $\alpha > 0$ is small and if $y$ is a function oscillating much faster than $\alpha$ Hz inside the interval $[-B,B]$ (e.g., $y(t) = \sin (2 \pi \beta t)$
for $-B \leq t \leq B$ where $\beta > \alpha$),
then a function $m \in PW [-\alpha,\alpha]$ satisfying \eqref{eqn:m-approximates-y} will have an extremely large magnitude outside the interval $[-B,B]$.
Indeed, to compensate for a fast oscillation exceeding the bandlimit,
the function $m$ necessarily exhibits large amplitudes outside the interval $[-B,B]$ leading to a huge energy cost $\| m \|_{L^2}^2 \gg 1$. This is a common feature of so-called \emph{superoscillations}; see for instance \cite{FK06,FK02,LF14-energy}.
\end{remark}

\begin{lemma}\label{lem:approx-by-HSOperator-input-chi}
Let $H \in \mathcal{L} (L^2(\R))$ be a multiplication operator with multiplier $m \in PW[-\alpha,\alpha]$ where $\alpha > 0$.
Then for any $B , \gamma, \epsilon'  > 0$, 
there exists an operator $\widetilde{H} \in OPW^2 ([-\gamma,\gamma] {\times} [-\alpha,\alpha])$ satisfying $\| (\widetilde{H} - H) \chi_{[-B,B]} \|_{L^2} < \epsilon'$.
\end{lemma}

\begin{proof}
Let $\widetilde{H}$ be the operator with spreading function $\eta_{\widetilde{H}} (t,\nu) = u(t) \, \widehat{m} (\nu)$ for some $u \in L^2(\R)$. For any $f \in L^2(\R)$, we have
\[
\begin{split}
\int \vert \widetilde{H} f (x) \vert^2 \, dx
&= \int \left\vert \iint u(t) \, \widehat{m} (\nu) \, e^{2 \pi i x \nu} \, f(x-t) \, dt \, d\nu \right\vert^2 \, dx \\
&= \int \vert m (x) \vert^2 \cdot  \left\vert  \int u(t) \, f(x-t) \, dt \right\vert^2  \, dx \\
&\leq \int \vert m (x) \vert^2 \cdot  \|u\|_{L^2}^2 \cdot \|f\|_{L^2}^2 \, dx
= \| m \|_{L^2}^2 \cdot  \|u\|_{L^2}^2 \cdot \|f\|_{L^2}^2 ,
\end{split}
\]
and since $\| \eta_{\widetilde{H}} \|_{L^2} = \| u \|_{L^2} \cdot \| m \|_{L^2} < \infty$, the operator $\widetilde{H}$ is in $\mathcal{L} (L^2(\R))$ and is Hilbert-Schmidt. 
Further, if $\supp u \subset [-\gamma,\gamma]$, then
\[
\begin{split}
&(\widetilde{H} - H) \chi_{[-B,B]} (x) \\
&= \iint \eta_{\widetilde{H}} (t,\nu) \, e^{2 \pi i x \nu} \, \chi_{[-B,B]} (x-t) \, dt \, d\nu
\;-\; H \chi_{[-B,B]} (x)  \\
&= \int \widehat{m} (\nu) \, e^{2 \pi i x \nu} \, d\nu \cdot  \int u(t) \, \chi_{[-B,B]} (x-t) \, dt
\;-\; m (x) \cdot \chi_{[-B,B]} (x)  \\
&= m (x) \cdot \left( u * \chi_{[-B,B]}  (x) - \chi_{[-B,B]} (x) \right)  .
\end{split}
\]
In particular, setting $u = \frac{1}{2\delta} \chi_{[-\delta,\delta]}$ with $0 < \delta < \gamma$ (one can instead use an approximate identity, i.e., a nonnegative function $\varphi_\delta \in \mathcal{S}(\R)$ with $\supp  \varphi_\delta \subset [-\delta,\delta]$ and $\int \varphi_\delta = 1$), we have
\[
u * \chi_{[-B,B]}  (x)
=
\begin{cases}
0 & \text{for} \;\; x \leq -B{-}\delta \, , \\
\frac{1}{2\delta}(x{+}B{+}\delta) & \text{for} \;\; {-}B{-}\delta \leq x \leq -B{+}\delta \, , \\
1 & \text{for} \;\; {-}B{+}\delta \leq x \leq B{-}\delta \, , \\
-\frac{1}{2\delta}(x{-}B{-}\delta) & \text{for} \;\; B{-}\delta \leq x \leq B{+}\delta \, , \\
0 & \text{for} \;\; x \geq B{+}\delta \, ,
\end{cases}
\]
so that
\[
\big\vert (\widetilde{H} - H) \chi_{[-B,B]} (x) \big\vert
\;\leq\;
\vert m (x) \vert \cdot \chi_{[-B-\delta,-B+\delta] \cup [B-\delta,B+\delta]} (x)
\]
and thus
\[
\big\| (\widetilde{H} - H) \chi_{[-B,B]} \big\|_{L^2}^2
\;\leq\;
\int_{[-B-\delta,-B+\delta] \cup [B-\delta,B+\delta]} \vert m (x) \vert^2 \, dx ,
\]
where the right hand side tends to zero as $\delta \rightarrow 0$
since $m  \in L^2(\R)$.
Choose a sufficiently small $\delta > 0$, so that the right hand side is smaller than $\epsilon'$.
Finally, since $\eta_{\widetilde{H}} (t,\nu) = \frac{1}{2\delta} \chi_{[-\delta,\delta]} (t) \, \widehat{m} (\nu)$ is supported in $[-\delta,\delta] {\times} [-\alpha,\alpha] \subset [-\gamma,\gamma] {\times} [-\alpha,\alpha]$ and since $\iint \vert \eta_{\widetilde{H}} (t,\nu) \vert^2 \, dt \, d\nu = \frac{1}{2\delta} \, \| m \|_{L^2}^2 < \infty$, the corresponding operator $\widetilde{H}$ is in $OPW^2 ([-\gamma,\gamma] {\times} [-\alpha,\alpha])$.
\end{proof}

Combining Proposition \ref{prop:existsMultiplicationOperator-input-chi} and Lemma \ref{lem:approx-by-HSOperator-input-chi} with $\epsilon' = \frac{1}{2} (1-c) \epsilon$, we obtain the desired result.

\begin{nonumber-theorem}[Theorem \ref{thm:existsHSOperator-input-chi} restated]
Let $\alpha , \epsilon , \gamma > 0$ and $y \in L^2(\R) \backslash \{ 0 \}$.
There exist a constant $B > 0$ and an operator $\widetilde{H} \in OPW^2 ([-\gamma,\gamma] {\times} [-\alpha,\alpha])$ such that $\| \widetilde{H} \chi_{[-B,B]} - y \|_{L^2} < \epsilon$.
Moreover, one can choose any $B > 0$ satisfying $\int_{[-B,B]^C} \vert y(x) \vert^2 \, dx < c \, \epsilon$ for some $0 < c < 1$; in particular, if $\supp y \subset [-M,M]$, then $B$ can be chosen to be $M$.
\end{nonumber-theorem}

\begin{remark}\label{rmk:first-main-result}
\rm
Taking $\eta_{\widetilde{H}} (t,\nu) = \frac{1}{2\delta} \chi_{[-\delta,\delta]}(t) \, \widehat{m} (\nu)$ as in the proof of Lemma \ref{lem:approx-by-HSOperator-input-chi}, we have
$\| \widetilde{H} \|_{HS}^2
= \| \eta_{\widetilde{H}} \|_{L^2(\R^2)}^2
= \frac{1}{2\delta} \, \| m \|_{L^2}^2$
which will be extremely large if $y \in L^2(\R)$ contains a fast oscillation inside $[-B,B]$ (see Remark \ref{rmk:prop:existsMultiplicationOperator-input-chi}).
Note that in general, the operator norm is dominated by the Hilbert-Schmidt norm, i.e., $\| \widetilde{H} \|_{L^2(\R) \rightarrow L^2(\R)} \leq \| \widetilde{H} \|_{HS}$.
However, Proposition \ref{prop:Thm4-1-in-KP14} indicates that the operator norm
$\| \widetilde{H} \|_{L^2(\R) \rightarrow L^2(\R)}
\, (\geq
A \, \| \sigma_{\widetilde{H}} \|_{L^\infty(\R^2)})$
should also be large; indeed, we have $\sigma_{\widetilde{H}}  (x,\xi) = m(x) \cdot \mathcal{F} ( \frac{1}{2\delta} \chi_{[-\delta,\delta]} )(\xi) = m(x) \, \sinc (2 \delta \xi)$ by \eqref{eqn:eta-sigma-relation} and thus $\| \sigma_{\widetilde{H}} \|_{L^\infty(\R^2)} = \| m \|_{L^\infty} \gg 1$.
\end{remark}

\section{Proof of Theorem \ref{thm:existsHSOperator-input-Sinc}}
\label{sec:thm:existsHSOperator-input-Sinc}

\begin{proposition}\label{prop:existsFourierMultiplier-inputSinc}
Let $\alpha , \epsilon > 0$, $0 < c < 1$, and $y \in L^2(\R) \backslash \{ 0 \}$.
There exist a constant $B > 0$ and a convolution operator $H \in \mathcal{L} (L^2(\R))$ with kernel $h \in L^2[-\alpha,\alpha]$ (meaning that $h \in L^2(\R)$ and $\supp h \subset [-\alpha,\alpha]$),
such that $\| H \phi_B - y \|_{L^2} < \frac{1}{2} (1+c) \epsilon$.
Moreover, one can choose any $B > 0$ satisfying $\int_{[-B,B]^C} \vert\widehat{y}(\xi)\vert^2 \, d\xi < c \, \epsilon$; in particular, if $\supp \widehat{y} \subset [-M,M]$, then $B$ can be chosen to be $M$.
\end{proposition}

\begin{proof}
Since $y \in L^2(\R)$, one can choose a sufficiently large $B > 0$ such that
$\int_{[-B,B]^C} \vert\widehat{y}(\xi)\vert^2 \, d\xi < c \, \epsilon$.
If $H \in \mathcal{L} (L^2(\R))$ is a convolution operator with kernel $h \in L^2(\R)$, then
\[
\begin{split} 
\| H \phi_B - y \|_{L^2} ^2
&= \| \widehat{H \phi_B} - \widehat{y} \|_{L^2} ^2
= \int \big\vert \widehat{h} (\xi) \, \chi_{[-B,B]} - \widehat{y}(\xi) \big\vert^2 \, d\xi \\
&= \int_{-B}^B \big\vert \widehat{h} (\xi) - \widehat{y}(\xi) \big\vert^2 \, d\xi + \int_{[-B,B]^C} \big\vert \widehat{y}(\xi) \big\vert^2 \, d\xi \\
&= \int_{-B}^B \big\vert \widehat{h} (\xi) - \widehat{y}(\xi) \big\vert^2 \, d\xi +  c \, \epsilon . 
\end{split}
\] 
Therefore, it is enough to find a function $h \in L^2[-\alpha,\alpha]$ (equivalently, a bandlimited function $\widehat{h} \in PW[-\alpha,\alpha]$) satisfying $\int_{-B}^B \vert \widehat{h} (\xi) - \widehat{y}(\xi) \vert^2 \, d\xi < \frac{1}{2} (1-c) \epsilon$. 
Such a function can be found similarly as in the proof of Proposition \ref{prop:existsMultiplicationOperator-input-chi}.
\end{proof}

\begin{remark}
\rm
Instead of the function $\phi_B$,
one could use any function $\psi_B \in L^2(\R)$ with $\supp \widehat{\psi}_B \subset [-B,B]$ and $\vert \widehat{\psi}_B (\xi) \vert \geq C > 0$ for a.e.~$\xi \in [-B,B]$. In that case, one simply needs to replace $\widehat{y}(\xi)$ with $\frac{\widehat{y}(\xi)}{\widehat{\psi}_B (\xi)}$ in the proof above.
\end{remark}

The following lemma can be proved similarly as Lemma \ref{lem:approx-by-HSOperator-input-chi}. 

\begin{lemma}\label{lem:approx-by-HSOperator-input-Sinc}
Let $H \in \mathcal{L} (L^2(\R))$ be a convolution operator with kernel $h \in L^2[-\alpha,\alpha]$ where $\alpha > 0$.
Then for any $B , \beta , \epsilon > 0$, 
there exists an operator $\widetilde{H} \in OPW([-\alpha,\alpha] {\times} [-\beta,\beta])$ satisfying $\| (\widetilde{H} - H) \phi_B \|_{L^2} <  \epsilon$. 
\end{lemma}

Combining Proposition \ref{prop:existsFourierMultiplier-inputSinc} and Lemma \ref{lem:approx-by-HSOperator-input-Sinc} gives the desired result.

\begin{nonumber-theorem}[Theorem \ref{thm:existsHSOperator-input-Sinc} restated]
Let $\alpha , \beta , \epsilon > 0$ and $y \in L^2(\R) \backslash \{ 0 \}$.
There exist a constant $B > 0$ and an operator $\widetilde{H} \in OPW^2([-\alpha,\alpha] {\times} [-\beta,\beta])$ such that $\| \widetilde{H} \phi_B - y \|_{L^2} < \epsilon$.
Moreover, one can choose any $B > 0$ satisfying $\int_{[-B,B]^C} \vert\widehat{y}(\xi)\vert^2 \, d\xi < c \, \epsilon$ for some $0 < c < 1$; in particular, if $\supp \widehat{y} \subset [-M,M]$, then $B$ can be chosen to be $M$.
\end{nonumber-theorem}

\backmatter

\bmhead{Acknowledgments}

The author acknowledges support by the DFG (German Research Foundation) Grants PF 450/6-1 and PF 450/9-1.
The author would also like to thank G\"{o}tz Pfander for valuable comments.

\section*{Declarations}

The author states that there is no conflict of interest.

\end{document}